\documentclass{article}
\usepackage{amsfonts}
\usepackage{latexsym}
\usepackage{mathrsfs}
\usepackage{amssymb}
\usepackage{amsmath}
\usepackage{amsthm}
\usepackage{indentfirst}
\usepackage{color}
\usepackage{float}
\usepackage{graphicx}

\allowdisplaybreaks
\newtheorem{theorem}{\color{black}\indent Theorem}[section]
\newtheorem{lemma}{\color{black}\indent Lemma}[section]

\newtheorem{definition}{\color{black}\indent Definition}[section]
\newtheorem{remark}{\color{black}\indent Remark}[section]

\begin{document}
\title{\LARGE\bf Blow-up phenomena for a reaction diffusion equation with special diffusion process}
\author{Yuzhu Han}
 \date{}
 \maketitle

\footnotetext{\hspace{-1.9mm}$^\dag$Corresponding author.\\
Email addresses: yzhan@jlu.edu.cn(Y. Han).

\thanks{
$^*$Supported by NSFC (11401252) and by The Education Department of Jilin Province (JJKH20190018KJ).}}
\begin{center}
{\noindent\it\small School of Mathematics, Jilin University, Changchun 130012, P.R. China}
\end{center}

\date{}
\maketitle

{\bf Abstract}\ This paper is concerned with the blow-up property of solutions to an initial boundary value
problem for a reaction diffusion equation with special diffusion processes.
It is shown, under certain conditions on the initial data, that the solutions to this problem blow up in finite time,
by combining Hardy inequality, ``moving" potential well methods with some differential inequalities.
Moreover, the upper and lower bounds for the blow-up time are also derived when blow-up occurs.



\section{Introduction}
\setcounter{equation}{0}

It is well known, by the conservation law, that many diffusion processes with reaction can be described by the following
equation (see \cite{Pao})
\begin{equation}\label{rd}
u_t-\nabla\cdot(D\nabla u)=f(x,t,u,\nabla u),
\end{equation}
where $u(x,t)$ stands for the mass concentration in chemical reaction processes or temperature in heat conduction,
at position $x$ in the diffusion medium and time $t$,
the function $D$ is called the diffusion coefficient or the thermal diffusivity,
the term $\nabla\cdot(D\nabla u)$ represents the rate of change due to diffusion and
$f(x,t,u,\nabla u)$ is the rate of change due to reaction.

In this paper, we consider the finite time blow-up properties of solutions to an
initial boundary value problem of \eqref{rd} with the special diffusion coefficient $D=|x|^2$,
i.e., to the following problem
\begin{equation}\label{p}
\begin{cases}
\dfrac{u_t}{|x|^2}-\Delta u=k(t)u^p, &(x,t)\in \Omega\times(0,T),\\
u(x,t)=0, &(x,t)\in \partial \Omega\times(0,T),\\
u(x,0)=u_0(x), & x\in\Omega,
\end{cases}
\end{equation}
where $\Omega$ is a bounded domain in $\mathbb{R}^n (n\geq3)$ containing the origin $0$ with smooth boundary $\partial \Omega$,
$1<p<\frac{n+2}{n-2}$, and the initial datum $u_0\in H_0^1(\Omega)$ is nonnegative and nontrivial. Moreover,
for $x=(x_1,x_2,\cdots,x_n)\in \mathbb{R}^n$, $|x|=\sqrt{x_1^2+x_2^2+\cdots+x_n^2}$.
The weight function $k(t)$ satisfies
$${\rm (A)}\qquad\qquad k(t)\in C^1[0,+\infty),\ k(0)>0,\ k'(t)\geq0,\qquad\qquad\qquad \forall\ t\in[0,+\infty).$$

In many practical situations one would like to know whether the solutions to some evolution problems blow up, and if so,
at which time $T$ blow-up occurs. Since $T$ can not be determined explicitly in most cases,
it is an important issue to establish the lower and (or) upper bounds for $T$.
There is an abounding literature on blow-up properties of solutions to nonlinear
evolution equations and systems, of which we only refer the reader to the monograph of Hu \cite{Hu},
and to the survey papers of Levine \cite{Levine} and of Galaktionov and V\'{a}zquez \cite{Gala}.
Clearly this list of references is far from complete and could be enlarged by the numerous papers cited in \cite{Gala,Hu,Levine}.

For the case $D\equiv1$ and $f(x,t,u,\nabla u)=u^p\ (p>1)$,
the blow-up phenomena for equation \eqref{rd} in both bounded domains and the whole space
have been studied extensively, which can also be found in the above mentioned references.
However, much less effort has been devoted to the blow-up properties of solutions to problem \eqref{p}.
When $k(t)\equiv1$, Tan \cite{Tan1} considered the existence and
asymptotic estimates of global solutions and finite time blow-up of local solutions to problem \eqref{p}.
By using the potential well method proposed by Sattinger and Payne \cite{Payne1975,Sattinger1968} and Hardy inequality,
he gave some sufficient conditions for the solutions to exist globally or to blow up in finite time,
when the initial energy is subcritical, i.e., initial energy smaller than the mountain pass level.
These results were later extended by Tan to $p$-Laplace equation with subcritical initial energy \cite{Tan2},
by Han to $p$-Laplace equation with supercritical initial energy \cite{Han2} and
by Zhou to porous medium equation and polytropic filtration equation \cite{Zhou1,Zhou2}.

Motivated by the works mentioned above, we shall consider the blow-up phenomena for problem \eqref{p}
and investigate what role the weight function $k(t)$ plays in determining the blow-up condition
and blow-up time of solutions to problem \eqref{p}.
To be a little more precise, we shall show, under the assumption (A) on $k(t)$, that the solutions to
problem \eqref{p} blow up in finite time if one of the following three assumptions holds:

(i) the initial energy is negative, i.e., $J(u_0;0)<0$;

(ii) the initial Nehari energy is negative and the initial energy is smaller than or equal to the potential well depth at infinity,
i.e., $I(u_0;0)<0$ and $J(u_0;0)\leq d(\infty)$;

(iii) $0<J(u_0;0)<C\|\frac{u_0}{|x|}\|_2^2$ for some positive constant $C$.\\
Moreover, the upper and lower bounds for the blow-up time are also derived,
with the help of Gagliardo-Nirenberg's inequality.
The main difficulties are of course caused by the weight function $k(t)$ and the singular coefficient $|x|^{-2}$.
Since $k(t)$ may not be a constant,
we have to consider the ``moving" potential wells, i.e., potential wells vary with time $t$, when proving case (ii).
To overcome the difficulty caused by $|x|^{-2}$, we apply Hardy inequality when dealing with case (iii).

The rest of this paper is organized as follows. In Section 2,
we shall introduce some definitions and auxiliary lemmas as preliminaries.
In Section 3, we give three sufficient conditions for the solutions to problem \eqref{p} to blow up in finite time,
and derive the upper bounds for blow-up time for each case. The lower bound for blow-up time will be derived in Section 4.

\section{Preliminaries}
\setcounter{equation}{0}

Throughout this paper, we denote by $\|\cdot\|_r$ the norm on $L^r(\Omega) (1\leq r\leq\infty)$,
and by $(\cdot ,\cdot )$ the inner product in $L^2(\Omega)$.
By $H_0^1(\Omega)$ we denote the Sobolev space such that both $u$ and $|\nabla u|$ belong to $L^2(\Omega)$ for any $u\in H_0^1(\Omega)$,
which will be endowed with the equivalent norm $\|u\|_{H_0^1(\Omega)}=\|\nabla u\|_2$.

We first recall a classical result essentially due to Hardy (see \cite{Hardy}).

\begin{lemma}\label{Hardy}
Assume that $u\in H^{1}(\mathbb{R}^n)$, $n\geq3$. Then $\dfrac{u}{|x|}\in L^2(\mathbb{R}^n)$, and
\begin{equation}\label{HI}
\int_{\mathbb{R}^n}\dfrac{|u|^2}{|x|^2}{\rm d}x\leq H_{n}\int_{\mathbb{R}^n}|\nabla u|^2{\rm d}x,
\end{equation}
where $H_{n}=4/(n-2)^2$.
\end{lemma}

\begin{remark}
For any $u\in H_0^{1}(\Omega)$, extend $u(x)$ to be $0$ for $x\in \mathbb{R}^n\setminus\Omega$.
Then $u\in H^{1}(\mathbb{R}^n)$ and therefore \eqref{HI} also holds for  $u\in H_0^{1}(\Omega)$.
\end{remark}

For any $u\in H_0^1(\Omega)$ and $t\geq0$, define the time-dependent energy functional and Nehari functional, respectively, by
\begin{equation}\label{j}
J(u;t)=\frac{1}{2}\|\nabla u\|_2^2-\frac{k(t)}{p+1}\|u\|^{p+1}_{p+1},
\end{equation}
and
\begin{equation}\label{i}
I(u;t)=\|\nabla u\|_2^2-k(t)\|u\|^{p+1}_{p+1}.
\end{equation}
Since $p+1<\frac{2n}{n-2}$, both $J(\cdot\ ;t)$ and $I(\cdot\ ;t)$ are well defined and continuous in $H_0^1(\Omega)$ for any $t\geq0$.
We also define, for any $t\geq0$, the ``moving" Nahari's manifold by
$$\mathcal{N}(t)=\{v\in H_0^1(\Omega)\setminus\{0\}:I(v;t)=0\}.$$
It is not hard to verify that $\mathcal{N}(t)$ is non-empty and the potential well depth
$$d(t)=\inf_{v\in H^{1}_{0}(\Omega)\atop v\neq0}\sup_{\lambda\geq0}J(\lambda v;t)=\inf_{v\in N(t)}J(v;t)$$
is positive for any $t\geq0$.

\begin{lemma}\label{depth}
Suppose that (A) holds. Then for any $t\in[0,\infty)$,

{\rm (1)}
\begin{equation}\label{d}
d(t)=\dfrac{p-1}{2(p+1)}\left (k(t)\right)^{\frac{2}{1-p}}S_{p}^{\frac{2(p+1)}{p-1}}>0,
\end{equation}
where
\begin{equation}\label{sp}
S_{p}=\inf_{v\in H^{1}_{0}(\Omega)\atop v\neq0}\frac{\|\nabla v\|_{2}}{\|v\|_{p+1}}.
\end{equation}

{\rm (2)} $d(t)$ is non-increasing and $d(\infty)\in [0,d(0)]$, where $d(\infty):=\underset{t\rightarrow\infty}{\lim}d(t)$.
\end{lemma}

\begin{proof}
(1) Fix $0\neq v\in H_0^1(\Omega)$ and $t\geq0$. Set
\begin{equation*}
F(\lambda):=J(\lambda v;t)=\frac{\lambda^{2}}{2}\|\nabla v\|^{2}_{2}-\frac{k(t)}{p+1}\lambda^{p+1}\|v\|^{p+1}_{p+1},\qquad\lambda\geq0.
\end{equation*}
Then it is easy to see that $F(\lambda)$ has only one critical point $\lambda_0=\left(\frac{\|\nabla v\|^{2}_{2}}{k(t)\|v\|^{p+1}_{p+1}}\right)^{\frac{1}{p-1}}$,
$F(\lambda)$ is increasing on $(0,\lambda_0)$, decreasing on $(\lambda_0,\infty)$ and takes its maximum at $\lambda=\lambda_0$.
Therefore,
\begin{align*}
d(t)&=\inf_{v\in H^{1}_{0}(\Omega)\atop v\neq0}\sup_{\lambda\geq0}J(\lambda v;t)=\inf_{v\in H^{1}_{0}(\Omega)\atop v\neq0}F(\lambda_0)\\
&=\inf_{v\in H^{1}_{0}(\Omega)\atop v\neq0}\left\{\frac{\lambda_{0}^{2}}{2}\|\nabla v\|^{2}_{2}-\frac{k(t)}{p+1}\lambda_{0}^{p+1}\|v\|^{p+1}_{p+1}\right\}\\
&=\frac{p-1}{2(p+1)}(k(t))^{\frac{2}{1-p}}\left(\inf_{v\in H^{1}_{0}(\Omega)\atop v\neq0}\frac{\|\nabla v\|_{2}}{\|v\|_{p+1}}\right)^{\frac{2(p+1)}{p-1}}\\
&=\frac{p-1}{2(p+1)}(k(t))^{\frac{2}{1-p}}S_{p}^{\frac{2(p+1)}{p-1}}.
\end{align*}

(2) From assumption (A) and \eqref{d} we know that the conclusions in (2) are valid. The proof is complete.
\end{proof}

In this paper, we consider weak solutions to problem \eqref{p}, which is defined as follows.

\begin{definition}
\label{ddef1}(See \cite{Tan1})
A function $u$ is called a (weak) solution to problem \eqref{p} in $\Omega\times (0,T)$ if
$$
u\in L^\infty(0,T;H_0^1(\Omega)),\qquad \int_0^T\Big\|\dfrac{u_t(t)}{|x|}\Big\|_2^2\rm{d}t<\infty,
$$
and $u(x,t)$ satisfies $u(x,0)=u_0(x)$ and
\begin{equation}\label{weak}
\Big(\dfrac{u_t}{|x|^2},v\Big)+(\nabla u,\nabla v)=k(t)(u^p,v),  \quad\forall\ v\in H_0^1(\Omega),\ t\in(0,T).
\end{equation}
\end{definition}

Local existence of weak solutions to problem \eqref{p} can be obtained by using the standard regularization method.
Interested reader may refer to \cite{Tan1,Tan2} for a similar proof.
Moreover, it follows from the weak maximum principle that $u(x,t)$ is nonnegative since $u_0(x)\geq0$ in $\Omega$.
If no confusion arises, we simply write $u(t)$ to denote the weak solution $u(x,t)$ to problem \eqref{p}.
From now on, we shall denote by $T^*\in[0,+\infty)$ the maximal existence time of $u(t)$, which is defined as follows.

\begin{definition}\label{blow-up}
Let $u(t)$ be a weak solution to problem \eqref{p}. We say that $u(t)$ blows up at a finite time $T_0$ provided that
$u(t)$ exists for all $t\in[0,T_0)$ and
\begin{equation}\label{blow}
\lim\limits_{t\rightarrow T_0}\Big\|\dfrac{u(t)}{|x|}\Big\|_2^2=+\infty.
\end{equation}
In this case, we say that the maximal existence time of $u(t)$ is $T_0$.
If \eqref{blow} does not happen for any finite $T_0$,
then $u(t)$ is said to be a global solution and the maximal existence time of $u(t)$ is $+\infty$.
\end{definition}

Let the assumption (A) hold and assume that $u(t)$ is a weak solution to problem \eqref{p}.
Then the following energy identity follows from a quite standard argument.
\begin{lemma}\label{monotone}
(\cite{Sun}) Let the assumption (A) hold and $u(t)$ be a weak solution to problem \eqref{p}.
Then $J(u(t);t)$ is non-increasing in $t$ and it holds, for any $t\in(0,T^*)$, that
\begin{equation}\label{ei}
J(u(t);t)+\int_0^t\Big(\Big\|\dfrac{u_\tau(\tau)}{|x|}\Big\|_2^2+\dfrac{k'(\tau)}{p+1}\|u(\tau)\|^{p+1}_{p+1}\Big){\rm d}\tau=J(u_0;0).
\end{equation}
\end{lemma}

Denote by $\mathcal{S}$ the set of weak solutions to the following elliptic problem
\begin{equation}\label{elliptic}
\begin{cases}
-\Delta w=k(0)|w|^{p-1}w,& x\in\Omega,\\
w(x)=0, &x\in\partial\Omega.
\end{cases}
\end{equation}
To show the finite time blow-up of solutions to problem \eqref{p} for subcritical initial energy,
we need some basic properties of $\mathcal{S}$ which are summarized into the following lemma.
Interested reader may refer to \cite{Sun} for a similar proof.

\begin{lemma}\label{s}
Assume that (A) holds and $u$ is a weak solution to problem \eqref{p} with initial datum $u_0\in H_0^1(\Omega)$. Then

{\rm (1)} $u_0\not\in\mathcal{S}$ provided that $\|\nabla u_{0}\|^{2}_{2}=\lambda_{1}\|u_{0}\|^{2}_{2}>0$,
where $\lambda_1>0$ is the first eigenvalue of \textcolor[rgb]{1.00,0.00,0.00}{$-\Delta$} in $\Omega$ under homogeneous Dirichlet boundary condition.

{\rm (2)} $\mathcal{S}\neq\emptyset$, $\mathcal{S}\subset \mathcal{N}(0)$ and $\mathcal{N}(0)\setminus\mathcal{S}\neq\emptyset$.

{\rm (3)} $\Big\|\dfrac{u_t(0)}{|x|}\Big\|_2^2>0$ provided that $u_{0}\not\in\mathcal{S}$.
\end{lemma}

We shall end up this section with the next two lemmas.
The first one is a special form of Gagliardo-Nirenberg's inequality (see \cite{Han2})
and the second one is the starting point when applying concavity argument \cite{Levine1973}.

\begin{lemma}\label{G-N}
Let $1<p<\frac{n+2}{n-2}$. Then for any $u\in H_0^{1}(\Omega)$ we have
\begin{equation}\label{2.3}
\|u\|^{p+1}_{p+1}\leq G\|\nabla u\|^{\alpha(p+1)}_2\|u\|^{(1-\alpha)(p+1)}_2,
\end{equation}
where $\alpha=\frac{n(p-1)}{2(p+1)}\in(0,1)$
and $G>0$ is a constant depending only on $\Omega$, $n$ and $p$.
\end{lemma}

\begin{lemma}\label{concave}(See \cite{Han2,Levine1973})
Suppose that a positive, twice-differentiable function $\psi(t)$ satisfies the inequality
$$\psi''(t)\psi(t)-(1+\theta)(\psi'(t))^2\geq0,\qquad\forall\ t\geq t_0\geq0,$$
where $\theta>0$. If $\psi(t_0)>0$, $\psi'(t_0)>0$, then $\psi(t)\rightarrow\infty$ as $t\rightarrow t_*\leq t^*=t_0+\dfrac{\psi(t_0)}{\theta\psi'(t_0)}$.
\end{lemma}

\section{Upper bound for blow-up time}
\setcounter{equation}{0}

With the preliminaries given in Section 2 at hand, we can now state and prove the main results in this paper.
For simplicity, we shall write $L(t)=\dfrac{1}{2}\Big\|\dfrac{u(t)}{|x|}\Big\|_2^2$ in the sequel.
We first prove a finite time blow-up result for problem \eqref{p} with negative initial energy.

\begin{theorem}\label{th1}
Let (A) hold and $u(t)$ be a weak solution to problem \eqref{p}. If $J(u_0;0)<0$, then
$u(t)$ blows up in finite time. Moreover, $T^*\leq \dfrac{2L(0)}{(1-p^2)J(u_0;0)}=\dfrac{\Big\|\dfrac{u_0}{|x|}\Big\|_2^2}{(1-p^2)J(u_0;0)}$.
\end{theorem}

\begin{proof}
We shall apply the first order differential inequality technique from Philippin \cite{Philipin}
to show the finite time blow-up result for problem \eqref{p} with negative initial energy.
For this, set $K(t)=-J(u(t);t)$. Then $L(0)>0$, $K(0)>0$.
From (A) and \eqref{ei} it follows that
$$K'(t)=-\dfrac{d}{dt}J(u(t);t)=\Big\|\dfrac{u_t(t)}{|x|}\Big\|_2^2+\dfrac{k'(t)}{p+1}\|u(t)\|^{p+1}_{p+1}\geq0,$$
which implies $K(t)\geq K(0)>0$ for all $t\in[0,T^*)$.
Recalling \eqref{j}, \eqref{i} and \eqref{weak}, we obtain, for any  $t\in[0,T^*)$, that
\begin{equation}\label{3.1}
\begin{split}
L'(t)=&\Big(\dfrac{u_t(t)}{|x|^2},u(t)\Big)=-I(u(t);t)=\dfrac{p-1}{2}\|\nabla u(t)\|_2^2-(p+1)J(u(t);t)\\
\geq&(p+1)K(t).
\end{split}
\end{equation}
Recalling \eqref{ei} and making use of Cauchy-Schwarz inequality, we arrive at
\begin{equation}\label{3.2}
L(t)K'(t)\geq\dfrac{1}{2}\Big\|\dfrac{u(t)}{|x|}\Big\|_2^2\Big\|\dfrac{u_t(t)}{|x|}\Big\|_2^2
\geq\dfrac{1}{2}\Big(\dfrac{u(t)}{|x|^2},u_t(t)\Big)^2=\dfrac{1}{2}(L'(t))^2\geq\dfrac{p+1}{2}L'(t)K(t),
\end{equation}
which then implies
\begin{equation*}
\Big(K(t)L^{-\frac{p+1}{2}}(t)\Big)'=L^{-\frac{p+3}{2}}(t)\Big(K'(t)L(t)-\dfrac{p+1}{2}K(t)L'(t)\Big)\geq0.
\end{equation*}
Therefore,
\begin{equation}\label{3.3}
0<\kappa:= K(0)L^{-\frac{p+1}{2}}(0)\leq K(t)L^{-\frac{p+1}{2}}(t)
\leq\dfrac{1}{p+1}L'(t)L^{-\frac{p+1}{2}}(t)=\dfrac{2}{1-p^2}\Big(L^{\frac{1-p}{2}}(t)\Big)'.
\end{equation}
Integrating \eqref{3.3} over $[0,t]$ for any $t\in(0,T^*)$ and noticing that $p>1$, one has
\begin{equation*}
\kappa t\leq \dfrac{2}{1-p^2}\Big(L^{\frac{1-p}{2}}(t)-L^{\frac{1-p}{2}}(0)\Big),
\end{equation*}
or equivalently
\begin{equation}\label{3.4}
0\leq L^{\frac{1-p}{2}}(t) \leq L^{\frac{1-p}{2}}(0)-\dfrac{p^2-1}{2}\kappa t,\qquad t\in(0,T^*).
\end{equation}
It is obvious that \eqref{3.4} can not hold for all $t>0$. Therefore, $T^*<+\infty$.
Moreover, it can be inferred from \eqref{3.4} that
$$T^*\leq \dfrac{2}{(p^2-1)\kappa}L^{\frac{1-p}{2}}(0)=\dfrac{2L(0)}{(1-p^2)J(u_0;0)}.$$
The proof is complete.
\end{proof}

\begin{remark}\label{rem1}
According to Theorem \ref{th1}, if the weak solution $u(t)$ to problem \eqref{p} exists globally,
then $J(u(t);t)\geq0$ for all $t\in[0,+\infty)$.
\end{remark}

For the case of $J(u_0;0)\geq0$, we obtain a blow-up results when the initial energy is ``subcritical" and when the initial Nehari functional
is negative. More precisely, we have the following theorem.

\begin{theorem}\label{th2}
Assume that (A) holds and that $u_0\in H_0^1(\Omega)$ satisfies
\begin{equation}\label{3.5}
J(u_0;0)\leq d(\infty)\qquad and \qquad\ I(u_0;0)<0.
\end{equation}
Then the weak solution $u(t)$ to problem \eqref{p} blows up in finite time.
Furthermore, $T^*$ can be estimated from above as follows
\begin{equation*}
T^*\leq t_0+\dfrac{8pL(t_0)}{(p+1)(p-1)^2[d(\infty)-J(u(t_0);t_0)]},
\end{equation*}
where $t_0\geq0$ is any finite time such that $J(u(t_0);t_0)<d(\infty)$.
In particular, if $J(u_0;0)<d(\infty)$, then
\begin{equation}\label{upper2}
T^*\leq \dfrac{4p\Big\|\dfrac{u_0}{|x|}\Big\|_2^2}{(p+1)(p-1)^2[d(\infty)-J(u_0;0)]}.
\end{equation}
\end{theorem}

\begin{proof}
We will divide the proof into three steps.

{\bf Step I.} Set $$V(t)=\{v\in H_0^1(\Omega):J(v;t)<d(t),I(v;t)<0\}\qquad t\in[0,T^*).$$
We claim that there exists a $t_0\in[0,T^*)$ such that $u(t)\in V(t)$ for all $t\in[t_0,T^*)$ provided that
$J(u_0;0)\leq d(\infty)$ and $I(u_0;0)<0$.

In fact, for the case of $J(u_0;0)<d(\infty)$, take $t_0=0$,
then it follows from \eqref{d} and \eqref{ei} that
\begin{equation}\label{3.7}
J(u(t);t)\leq J(u_{0};0)<d(\infty)\leq d(t),\qquad\ t\in[0,T^*).
\end{equation}
It remains to show that $I(u(t);t)<0$ for all $t\in[0,T^*)$. Since $I(u_0;0)<0$, by continuity, there exists a suitably small $t_1>0$ such that
$I(u(t);t)<0$ for all $t\in[0,t_1)$. Suppose on the contrary that
there exists a $t_2>t_1$ such that $I(u(t_2);t_2)=0$ and $I(u(t);t)<0$ for all $t\in[0,t_2)$.
Then by the definition of $d(t)$, one obtains
$$
J(u(t_{2});t_{2})\geq\underset{v\in \mathcal{N}(t_{2})}{\inf}J(v;t_{2})=d(t_{2}),
$$
which contradicts \eqref{3.7}. Therefore, $I(u(t);t)<0$ for all $t\in[0,T^*)$.

When $J(u_0;0)=d(\infty)$, by continuity and $I(u_0;0)<0$ we see that there exists a $t_3>0$ such that $I(u(t);t)<0$ for all $t\in[0,t_3)$.
In addition, Lemma \ref{s} says that $u_0$ is not a weak solution to problem \eqref{elliptic} and $\Big\|\dfrac{u_t(0)}{|x|}\Big\|_2^2>0$.
By continuity again, there exists a $t_0\in(0,t_3)$ such that $I(u(t_0);t_0)<0$ and $\Big\|\dfrac{u_t(t)}{|x|}\Big\|_2^2>0$ for all $t\in[0,t_0)$.
Therefore, by recalling \eqref{ei} and Lemma \ref{depth}, one obtains
\begin{equation}\label{3.8}
J(u(t);t)\leq J(u(t_{0});t_{0})\leq J(u_{0};0)-\int_{0}^{t_{0}}\Big\|\dfrac{u_{\tau}(\tau)}{|x|}\Big\|_2^2{\rm d}\tau <J(u_{0};0)=d(\infty)\leq d(t).
\end{equation}
By applying the argument similar to the case of $J(u_0;0)<d(\infty)$, we can show that $I(u(t);t)<0$ for all $t\in[t_0,T^*)$,
and therefore $u(t)\in V(t)$ for all $t\in[t_0,T^*)$. Moreover, since $L'(t)=-I(u(t);t)$, $L(t)$ is strictly increasing on $[t_0,T^*)$.

{\bf Step II.} We show that
\begin{equation}\label{3.9}
\|\nabla u(t)\|_2^2\geq\dfrac{2(p+1)d(t)}{p-1},\qquad\ t\in[t_0,T^*).
\end{equation}

From Step I we know that $I(u(t);t)<0$ for all $t\in[t_0,T^*)$. Therefore,
$$\|\nabla u(t)\|_2^2<k(t)\|u(t)\|^{p+1}_{p+1}\leq\dfrac{k(t)}{S_p^{p+1}}\|\nabla u(t)\|_2^{p+1},\qquad  t\in[t_0,T^*),$$
which implies \eqref{3.9}, by the definition of $d(t)$.

{\bf Step III.} We show that $T^*<+\infty$. For any $T\in(t_0,T^*)$,
define the positive function
\begin{equation}\label{F}
F(t)=\int_{t_0}^tL(\tau)\mathrm{d}\tau+(T-t)L(t_0)+\dfrac{\beta}{2}(t+\sigma)^2,\qquad t\in[t_0,T],
\end{equation}
where $\beta>0$ and $\sigma>-t_0$.

By direct computations
\begin{equation}\label{3.11}
\begin{split}
F'(t)=&L(t)-L(t_0)+\beta(t+\sigma)=\int_{t_0}^t\frac{d}{d\tau}L(\tau)\mathrm{d}\tau+\beta(t+\sigma)\\
=&\int_{t_0}^t\Big(u(\tau),\dfrac{u_\tau(\tau)}{|x|^2}\Big)\mathrm{d}\tau+\beta(t+\sigma),
\end{split}
\end{equation}
\begin{equation}\label{3.12}
\begin{split}
F''(t)=&L'(t)+\beta=\Big(u(t),\dfrac{u_t(t)}{|x|^2}\Big)+\beta=-I(u(t);t)+\beta\\
=&\dfrac{p-1}{2}\|\nabla u(t)\|_2^2-(p+1)J(u(t);t)+\beta\\
=&\dfrac{p-1}{2}\|\nabla u(t)\|_2^2-(p+1)\Big[J(u(t_0);t_0)-\int_{t_0}^t\Big(\Big\|\dfrac{u_\tau(\tau)}{|x|}\Big\|_2^2+\dfrac{k'(\tau)}{p+1}\|u(\tau)\|^{p+1}_{p+1}\Big){\rm d}\tau\Big]+\beta\\
\geq&\dfrac{p-1}{2}\|\nabla u(t)\|_2^2-(p+1)J(u(t_0);t_0)+(p+1)\int_{t_0}^t\Big\|\dfrac{u_\tau(\tau)}{|x|}\Big\|_2^2{\rm d}\tau+\beta.
\end{split}
\end{equation}
Applying Cauchy-Schwarz inequality and H\"{o}lder's inequality to yield
\begin{equation*}
\begin{split}
f(t):=&\Big[\int_{t_0}^t\Big\|\dfrac{u(\tau)}{|x|}\Big\|^2_2\mathrm{d}\tau
+\beta(t+\sigma)^2\Big]\Big[\int_{t_0}^t\Big\|\dfrac{u_\tau(\tau)}{|x|}\Big\|^2_2\mathrm{d}\tau+\beta\Big]
-\Big[\int_{t_0}^t\Big(u,\dfrac{u_\tau}{|x|^2}\Big)\mathrm{d}\tau+\beta(t+\sigma)\Big]^2\\
=&\Big[\int_{t_0}^t\Big\|\dfrac{u(\tau)}{|x|}\Big\|^2_2\mathrm{d}\tau\int_{t_0}^t\Big\|\dfrac{u_\tau(\tau)}{|x|}\Big\|^2_2\mathrm{d}\tau
-\Big(\int_{t_0}^t\Big(u,\dfrac{u_\tau}{|x|^2}\Big)\mathrm{d}\tau\Big)^2\Big]\\
&+\beta\Big[(t+\sigma)^2\int_{t_0}^t\Big\|\dfrac{u_\tau(\tau)}{|x|}\Big\|^2_2\mathrm{d}\tau+\int_{t_0}^t\Big\|\dfrac{u(\tau)}{|x|}\Big\|^2_2\mathrm{d}\tau
-2(t+\sigma)\int_{t_0}^t\Big(u,\dfrac{u_\tau}{|x|^2}\Big)\mathrm{d}\tau\Big]\\
\geq&0.
\end{split}
\end{equation*}
Therefore, by recalling \eqref{3.11}, \eqref{3.12} and noticing the nonnegativity of $f(t)$, we arrive at
\begin{equation}\label{3.13}
\begin{split}
&F(t)F''(t)-\frac{p+1}{2}(F'(t))^2\\
=&F(t)F''(t)+\frac{p+1}{2}\Big[f(t)-\Large[2F(t)-2(T-t)L(t_0)\Large]\big(\int_{t_0}^t\Big\|\dfrac{u_\tau}{|x|}\Big\|^2_2\mathrm{d}\tau+\beta\big)\Big]\\
\geq&F(t)F''(t)-(p+1)F(t)\Big(\int_{t_0}^t\Big\|\dfrac{u_\tau}{|x|}\Big\|^2_2\mathrm{d}\tau+\beta\Big)\\
\geq&F(t)\Big[\dfrac{p-1}{2}\|\nabla u(t)\|_2^2-(p+1)J(u(t_0);t_0)+(p+1)\int_{t_0}^t\Big\|\dfrac{u_\tau}{|x|}\Big\|^2_2\mathrm{d}\tau+\beta\\
&-(p+1)\int_{t_0}^t\Big\|\dfrac{u_\tau}{|x|}\Big\|^2_2\mathrm{d}\tau-(p+1)\beta\Big]\\
=&F(t)\Big[\dfrac{p-1}{2}\|\nabla u(t)\|_2^2-(p+1)J(u(t_0);t_0)-p\beta\Big].
\end{split}
\end{equation}

In view of \eqref{3.8}, \eqref{3.9} and \eqref{3.13}, we get,
for any $t\in[t_0,T]$ and $\beta\in\Big(0,\dfrac{(p+1)(d(\infty)-J(u(t_0);t_0))}{p}\Big]$ that
\begin{equation*}
F(t)F''(t)-\frac{p+1}{2}(F'(t))^2\geq0,\qquad t\in[t_0,T].
\end{equation*}
Therefore, Lemma \ref{concave} guarantees that
$$0<T-t_0\leq\dfrac{2F(t_0)}{(p-1)F'(t_0)}=\dfrac{2L(t_0)}{(p-1)\beta(t_0+\sigma)}(T-t_0)+\dfrac{t_0+\sigma}{p-1},$$
or
\begin{equation}\label{3.14}
(T-t_0)\Big(1-\dfrac{2L(t_0)}{(p-1)\beta(t_0+\sigma)}\Big)\leq \dfrac{t_0+\sigma}{p-1}.
\end{equation}

Fix a $\beta_0\in\Big(0,\dfrac{(p+1)(d(\infty)-J(u(t_0);t_0))}{p}\Big]$.
Then for any $\sigma\in\Big(\dfrac{2L(t_0)}{(p-1)\beta_0}-t_0,+\infty\Big)$, we have $0<\dfrac{2L(t_0)}{(p-1)\beta_0(t_0+\sigma)}<1$,
which, together with \eqref{3.14}, implies that
\begin{equation}\label{3.15}
T\leq t_0+\dfrac{t_0+\sigma}{p-1}\Big(1-\dfrac{2L(t_0)}{(p-1)\beta_0(t_0+\sigma)}\Big)^{-1}=t_0+\dfrac{\beta_0(t_0+\sigma)^2}{(p-1)\beta_0(t_0+\sigma)-2L(t_0)}.
\end{equation}
Minimizing the right hand side in \eqref{3.15} for $\sigma\in\Big(\dfrac{2L(t_0)}{(p-1)\beta_0}-t_0,+\infty\Big)$ to yield
\begin{equation}\label{3.16}
T\leq \inf\limits_{\sigma\in(\frac{2L(t_0)}{(p-1)\beta_0}-t_0,+\infty)}\Big[t_0+\dfrac{\beta_0(t_0+\sigma)^2}{(p-1)\beta_0(t_0+\sigma)-2L(t_0)}\Big]=
t_0+\dfrac{8L(t_0)}{(p-1)^2\beta_0}.
\end{equation}
Minimizing the right hand side of \eqref{3.16} with respect to $\beta_0\in\Big(0,\dfrac{(p+1)(d(\infty)-J(u(t_0);t_0))}{p}\Big]$
one obtains
\begin{equation*}
T\leq t_0+\dfrac{8pL(t_0)}{(p+1)(p-1)^2[d(\infty)-J(u(t_0);t_0)]}.
\end{equation*}
By the arbitrariness of $T<T^*$ we finally get
$$T^*\leq t_0+\dfrac{8pL(t_0)}{(p+1)(p-1)^2[d(\infty)-J(u(t_0);t_0)]}.$$
In particular, if $J(u_0;0)<d(\infty)$, then, by taking $t_0=0$, we have
$$T^*\leq \dfrac{4p\Big\|\dfrac{u_0}{|x|}\Big\|_2^2}{(p+1)(p-1)^2[d(\infty)-J(u_0;0)]}.$$
The proof is complete.
\end{proof}

\begin{remark}\label{rem2}
It is easily seen from \eqref{j} and \eqref{i} that $J(u_0;0)<0$ implies $I(u_0;0)<0$.
Therefore, Theorem \ref{th1} can be viewed as a special case of Theorem \ref{th2}.
But we obtained the upper bounds for $T^*$ by using different techniques.
By comparing the two upper bounds directly one can see that the one in Theorem \ref{th1} is more accurate when $0\leq d(\infty)\leq\frac{3p+1}{p-1}(-J(u_0;0))$,
while the one in Theorem \ref{th2} is more accurate when $0<\frac{3p+1}{p-1}(-J(u_0;0))\leq d(\infty)$.
\end{remark}

At the end of this section, we give another blow-up condition for problem \eqref{p},
which ensures that problem \eqref{p} admits blow-up solutions at arbitrarily high initial energy level.
The result in this direction is the following theorem.

\begin{theorem}\label{th3}
Assume that (A) holds and that $u(t)$ is a weak solution to problem \eqref{p}.
If
\begin{equation}\label{3.17}
0<J(u_0;0)<\dfrac{L(0)}{C_1},
\end{equation}
then $u(t)$ blows up at some finite time $T^*$.
Moreover, the upper bound for $T^*$ has the following form
$$T^*\leq \dfrac{8pH_nL(0)}{(p-1)^3[L(0)-C_1J(u_0;0)]},$$
where $C_1=\frac{(p+1)H_n}{p-1}$ and $H_n$ is the positive constant given in Hardy inequality.
\end{theorem}

\begin{proof}
This theorem will be proved by using some ideas from \cite{Han2,Sun} and an application of Hardy inequality.

First, by using \eqref{3.12} and Hardy inequality \eqref{HI} we have
\begin{equation}\label{3.18}
\begin{split}
\frac{d}{dt}L(t)=&\dfrac{p-1}{2}\|\nabla u(t)\|_2^2-(p+1)J(u(t);t)\\
\geq&\dfrac{p-1}{2H_n}\Big\|\frac{u(t)}{|x|}\Big\|_2^2-(p+1)J(u(t);t)\\
=&\dfrac{p-1}{H_n}\Big[L(t)-C_1J(u(t);t)\Big].
\end{split}
\end{equation}
Set
$$M(t)=L(t)-C_1J(u(t);t),\qquad t\in[0,T^*),$$
then $M(0)=L(0)-C_1J(u_0;0)>0$ by \eqref{3.17}.
Moreover, from \eqref{ei} and \eqref{3.18} it follows
\begin{equation}\label{3.19}
\frac{d}{dt}M(t)=\frac{d}{dt}L(t)-C_1\frac{d}{dt}J(u(t);t)\geq \frac{d}{dt}L(t)\geq\dfrac{p-1}{H_n}M(t).
\end{equation}
Therefore, an application of Gronwall's inequality implies that
\begin{equation}\label{3.20}
M(t)\geq M(0)e^{\frac{p-1}{H_n}t}>0,
\end{equation}
which, together with \eqref{3.18}, shows that $L(t)$ is strictly increasing on $[0,T^*)$.

For any $T\in(0,T^*)$, $\beta>0$ and $\sigma>0$,  define
$$F_1(t)=\int_{0}^tL(\tau)\mathrm{d}\tau+(T-t)L(0)+\dfrac{\beta}{2}(t+\sigma)^2,\qquad t\in[0,T].$$
Similarly to the derivation of \eqref{3.13} we get
\begin{equation}\label{3.21}
F_1(t)F_1''(t)-\frac{p+1}{2}(F_1'(t))^2\geq F_1(t)\Big[\dfrac{p-1}{2}\|\nabla u(t)\|_2^2-(p+1)J(u_0;0)-p\beta\Big].
\end{equation}
Applying Hardy inequality \eqref{HI} again and noticing the monotonicity of $L(t)$, we further obtain
\begin{equation}\label{3.22}
\begin{split}
F_1(t)F_1''(t)-\frac{p+1}{2}(F_1'(t))^2\geq& F_1(t)\Big[\dfrac{p-1}{2H_n}\Big\|\frac{u(t)}{|x|}\Big\|_2^2-(p+1)J(u_0;0)-p\beta\Big]\\
\geq&F_1(t)\Big[\dfrac{p-1}{2H_n}\Big\|\frac{u_0}{|x|}\Big\|_2^2-(p+1)J(u_0;0)-p\beta\Big]\\
=&\dfrac{p-1}{H_n}F_1(t)\Big[M(0)-\dfrac{pH_n}{p-1}\beta\Big]\geq0,
\end{split}
\end{equation}
for all $\beta\in(0,\frac{(p-1)M(0)}{pH_n}]$.

Starting with \eqref{3.22}, recalling Lemma \ref{concave} and applying similar arguments to that in the proof of Theorem \ref{th2}
we get
$$T^*\leq \dfrac{8pH_nL(0)}{(p-1)^3M(0)}=\dfrac{8pH_nL(0)}{(p-1)^3[L(0)-C_1J(u_0;0)]}.$$
The proof is complete.
\end{proof}

\begin{remark}\label{rem3}
Theorem \ref{th3} implies that for any $R>0$, there exists a $u_0$ such that $J(u_0;0)=R<L(0)/C_1$,
while the corresponding solution $u(x,t)$ to problem \eqref{p} with $u_0$ as initial datum blows up in finite time.
We refer the interested reader to \cite{Han2,Sun} for the standard proof of this statement.
\end{remark}

\section{Lower bound for blow-up time}
\setcounter{equation}{0}

In this section, we shall derive a lower bound for the blow-up time $T^*$,
by combining the famous Gagliardo-Nirenberg's inequality with the first order differential inequalities.

\begin{theorem}\label{lower}
Assume that (A) holds and $1<p<1+\dfrac{4}{n}$.  Let $u(t)$ be a weak solution to problem \eqref{p} that blows up at $T^*$.
Then $T^*\geq \dfrac{L^{1-\gamma}(0)}{C^*(\gamma-1)}$,
where $\gamma>1$ and $C^*>0$ are two constants that will be determined in the proof.
\end{theorem}

\begin{proof}
Combining \eqref{3.12} with Gagliardo-Nirenberg's inequality and recalling the monotonicity of $k(t)$, we have
\begin{equation}\label{4.1}
\begin{split}
L'(t)=&-I(u(t);t)=k(t)\|u(t)\|^{p+1}_{p+1}-\|\nabla u(t)\|_2^2\\
\leq&k_1G\|\nabla u(t)\|^{\alpha(p+1)}_2\|u(t)\|^{(1-\alpha)(p+1)}_2-\|\nabla u(t)\|_2^2,
\end{split}
\end{equation}
where $k_1$ is an arbitrary upper bound for $k(T^*)$, $G$ and $\alpha$ are the positive constants given in Lemma \ref{G-N}.
Since $1<p<1+\dfrac{4}{n}$, it is directly verified that
$$0<\alpha(p+1)=\dfrac{n(p-1)}{2}<2.$$
Applying Young's inequality to the first term on the right hand side of \eqref{4.1}, we obtain, for any $\varepsilon>0$, that
\begin{equation}\label{4.2}
\|\nabla u(t)\|^{\alpha(p+1)}_2\|u(t)\|^{(1-\alpha)(p+1)}_2\leq\dfrac{\alpha(p+1)}{2}\varepsilon\|\nabla u(t)\|^{2}_2+
\dfrac{2-\alpha(p+1)}{2}\varepsilon^{-\frac{\alpha(p+1)}{2-\alpha(p+1)}}\|u(t)\|^{2\gamma}_2,
\end{equation}
where $\gamma=\dfrac{(1-\alpha)(p+1)}{2-\alpha(p+1)}>1$.
Taking $\varepsilon=\dfrac{2}{k_1G\alpha(p+1)}$ and substituting \eqref{4.2} into \eqref{4.1} to yield
\begin{equation}\label{4.3}
L'(t)\leq C_2\|u(t)\|_2^{2\gamma}\leq2^\gamma (diam(\Omega))^{2\gamma}C_2L^\gamma(t):=C^*L^\gamma(t),
\end{equation}
where $C_2$ is a positive constant depending on $n$, $p$, $k_1$ and $G$, $diam(\Omega)>0$ is the diameter of $\Omega$
and $C^*=2^\gamma (diam(\Omega))^{2\gamma}C_2$. Integrating \eqref{4.3} over $[0,t)$, we get
\begin{equation*}
\dfrac{1}{1-\gamma}\Big\{L^{1-\gamma}(t)-L^{1-\gamma}(0)\Big\}\leq C^*t.
\end{equation*}
Since $\gamma>1$, letting $t\to T^*$ in the above inequality and recalling that $\lim\limits_{t\to T^*}L(t)=+\infty$,
we obtain
$$T^*\geq\dfrac{L^{1-\gamma}(0)}{C^*(\gamma-1)}.$$
The proof is complete.
\end{proof}

\begin{remark}
In \cite{Sun}, the authors investigated the blow-up properties of solutions to a class of semilinear parabolic or pseudo-parabolic equations,
and obtained, among many other interesting results, the lower bounds for the blow-up time only for the pseudo-parabolic case.
In our paper, by applying the famous Gagliardo-Nirenberg's inequality, we derived a lower bound for the blow-up time for the parabolic problem \eqref{p}.
Moreover, our treatment can also be applied to the parabolic problem considered in \cite{Sun} to obtain the lower bound for the blow-up time.
\end{remark}

\section*{Acknowledgements}
The author would like to express his sincere gratitude to Professor Wenjie Gao in Jilin University for his enthusiastic
guidance and constant encouragement. He would also like to thank the referees for their valuable
comments and suggestions, especially for pointing out the mistake of the blow-up time in Theorem \ref{th3} in the original manuscript.

{\small

\end{document}